\documentclass[twoside, 11pt, leqno]{amsart}
 \usepackage{amsmath, amssymb, amsfonts, amsthm, amscd} \usepackage{mathrsfs}
 \usepackage[pdftex]{graphicx}

\theoremstyle{plain}
\numberwithin{equation}{section}

\newtheorem{theorem}{Theorem}
\newtheorem{lemma}{Lemma}
\newtheorem{corollary}{Corollary}
\newtheorem{proposition}{Proposition}
\newtheorem{remark}{Remark}

\newcommand{\fracsm}[2]{\begin{matrix}\frac{#1}{#2}\end{matrix}}

\newcommand{\beq}{\begin{equation}}
\newcommand{\eeq}{\end{equation}}
\newcommand{\Reals}{\mathbb{R}}

\begin{document}

\title{Self-shrinkers with a rotational symmetry}
\author{Stephen J. Kleene}
\address{Stephen J. Kleene, Department of Mathematics, Johns Hopkins University, Baltimore, MD.}
\email{skleene@math.jhu.edu}
\author{Niels Martin M\o{}ller}
\address{Niels Martin M\o{}ller, Department of Mathematics, MIT, Cambridge, MA.}
\email{moller@math.mit.edu}
%\date{8th of September 2011}

\begin{abstract}
In this paper we present a new family of non-compact properly embedded, self-shrinking, asymptotically conical, positive mean curvature ends $\Sigma^n\subseteq\Reals^{n+1}$ that are hypersurfaces of revolution with circular boundaries. These hypersurface families interpolate between the plane and half-cylinder in $\Reals^{n+1}$, and any rotationally symmetric self-shrinking non-compact end belongs to our family. The proofs involve the global analysis of a cubic-derivative quasi-linear ODE.

We also prove the following classification result: a given complete, embedded, self-shrinking hypersurface of revolution $\Sigma^n$ is either a hyperplane $\Reals^{n}$, the  round cylinder $\Reals\times S^{n-1}$ of radius $\sqrt{2(n-1)}$, the round sphere $S^n$ of radius $\sqrt{2n}$, or is diffeomorphic to an $S^1\times S^{n-1}$ (i.e. a "doughnut" as in \cite{Ang}, which when $n=2$ is a torus). In particular for self-shrinkers there is no direct analogue of the Delaunay unduloid family. The proof of the classification uses translation and rotation of pieces, replacing the method of moving planes in the absence of isometries.
\end{abstract}

\maketitle

\section{Introduction and statement of results}
We consider smooth $n$-dimensional hypersurfaces $\Sigma^n\subseteq\Reals^{n+1}$, $n\geq 2$, possibly with boundary $\partial\Sigma\neq\emptyset$, satisfying the (extinction time $T=1$) self-shrinker equation for mean curvature flow, away from $\partial\Sigma$,
\beq\label{SSEq}
H=\frac{\langle \vec{X},\vec{\nu}\rangle}{2},
\eeq
where $\vec{\nu}$ is the unit normal such that $\vec{H}= - H\vec{\nu}$.

\begin{theorem}\label{ThmConicalEndsWordyVersion}
In $\Reals^{n+1}$ there exists a 1-parameter family of non-compact smooth rotationally symmetric, embedded, positive mean curvature, asymptotically conical self-shrinking ends $\Sigma^n$ with boundary.

In fact for each rotationally symmetric cone $\mathscr{C}$ in $\{x_1\geq 0\}\subseteq\Reals^{n+1}$ with tip at the origin, of slope $\sigma>0$, there is a unique such self-shrinker $\Sigma_\sigma$, lying outside of $\mathscr{C}$, which is asymptotic to $\mathscr{C}$ as $x_1\to\infty$.
\end{theorem}

\begin{theorem}\label{ThmUni}
Let $\Sigma^n\subseteq\Reals^{n+1}$ be a complete, embedded, self-shrinking hypersurface of revolution. Then $\Sigma^n$ is one of the following:
\begin{itemize}
\item[(1)] $n$-dimensional hyperplane $\Reals^n$ in $\Reals^{n+1}$,
\item[(2)] round cylinder $\Reals\times S^{n-1}$ of radius $\sqrt{2(n-1)}$,
\item[(3)] round sphere $S^n $of radius $\sqrt{2n}$,
\item[(4)] a smooth embedded $S^1\times S^{n-1}$.
\end{itemize}
\end{theorem}

\begin{remark}
Note that the list (1)--(3) together with Angenent's torus (in $\Reals^3$, and more general his specific $S^1\times S^{n-1}$-solutions) gives all the presently known examples of complete, embedded self-shrinkers.

In case (4), our assertion is only that $\Sigma$ is generated by a closed, smooth, embedded curve. We conjecture however that geometrically such $\Sigma$ must be symmetric with respect to $x_1\mapsto -x_1$ and in fact coincide with Angenent's torus in \cite{Ang}.
\end{remark}

By combining Theorem \ref{ThmUni} with a result by Anciaux \cite{Anc} we have the following corollary in $3$-space.

\begin{corollary}
Let $\Sigma^2\subseteq\Reals^3$ be a self-shrinker of genus zero which is foliated by circles. Then $\Sigma$ is either: a plane, a round cylinder of radius $\sqrt{2}$, or a round sphere of radius $2$.
\end{corollary}

\begin{figure}
\includegraphics[height = 300pt]{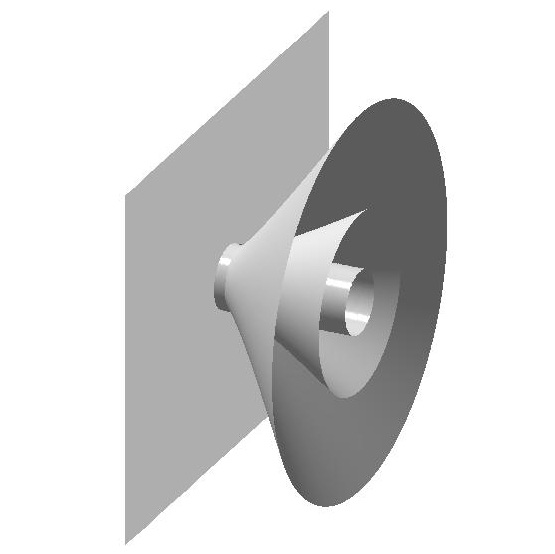}
\caption{Examples of the asymptotically conical self-shrinking ``trumpet'' ends in Theorem \ref{ThmConicalEndsWordyVersion}, interpolating between the flat plane and round cylinder (Matlab).}
\end{figure}

The study of the self-shrinker equation $H = \frac{1}{2} \langle \vec{X}, \vec{\nu} \rangle$ is motivated by the regularity theory for mean curvature flow. In particular, type I singularities are governed by (\ref{SSEq}), as Huisken showed in  \cite{Hu1}. Huisken in \cite{Hu3} classified the possible singularities for the flow of a positive mean curvature initial surface, under the assumption of a bound on $|A|^2$. Currently, very few complete solutions of Equation (\ref{SSEq}) are known, embedded or otherwise, with the sphere, plane, cylinder, and Angenent's Torus (constructed in \cite{Ang}) being the only known examples. There is however numerical evidence for many more. David Chopp in \cite{Ch} (and see \cite{AIC}) numerically computed a large number of interesting (apparently) self-similar solutions, and Angenent in \cite{Ang} gave numerical evidence for immersed  topological spheres,  although none of them have actually been rigorously demonstrated.  The methods in \cite{Ka} of Kapouleas for producing examples of complete embedded minimal surfaces in Euclidean space by desingularization promise to be successful in the context of Equation (\ref{SSEq}); in particular, X. H. Nguyen in \cite{Ng} has had success in providing examples of self-translating (not self-shrinking) surfaces under mean curvature flow.

The numerical evidence cited above suggests that, in general, the singularity profile for mean curvature flow can be quite exotic and wildly behaved, and the classification of solutions to ({\ref{SSEq}) seems impossible in general, however in dimension $2$ the methods of Colding-Minicozzi  in \cite{CM1}--\cite{CM7} offer a possibility. However, for a generic initial surface, one expects to find a rather tame singularity singularity profile, due to the highly unstable nature of  most solutions of (\ref{SSEq}). In fact, this is a long-standing conjecture of Huisken, which was recently answered positively by Colding and Minicozzi in \cite{CM7}.

The study of Equation (\ref{SSEq}) turns out to be a variational problem. Namely, the solutions are actually minimal hypersurfaces in the conformal metric (see \cite{Hu1})
\beq 
g = e^{-\frac{|\vec{X}|^2}{2n}} \sum_{i = 1}^{n+1} dx_i^2
\eeq
on $\Reals^{n+1}$. If $\Sigma_\gamma$ is a hypersurface of revolution determined by a planar curve $\gamma$, then $\Sigma_\gamma$ is minimal in the metric $g$ if and only if the curve $\gamma$ is a geodesic in the upper half-plane with non-complete conformal metric (c.f. \cite{Ang})
\beq \label{ang_metric}
g_\textrm{Ang}=r^{2(n-1)}e^{-\frac{(x^2 + r^2)}{2}}\left\{dx^2 + dr^2 \right\},
\eeq
where $(x,r)$, $r > 0$ are Euclidean coordinates on the upper half-plane. The idea of reducing the  problem of finding closed minimal sub-manifolds to the search for closed geodesics  on a related manifold with a singular metric goes back at least as far as the  paper \cite{HL}, where Hsiang and Lawson produced closed  minimal submanifolds of $S^n$ invariant under a subgroup of the full isometry group. Mean curvature flow restricted to the rotational class is not a new venture either. For example, in addition to \cite{Ang} the paper \cite{AAG95} considered regularity of viscosity solutions for mean curvature flow within the class of rotational surfaces. 

The geodesic equation for curves parametrized by arc length in the upper half plane with metric $g_{\textrm{Ang}}$ as given above in (\ref{ang_metric}) is given by the following system of equations (see \cite{Ang}):
\beq \label{ode_system}
\begin{cases}
&\dot{x} = \cos\theta \\
&\dot {r} = \sin\theta \\
&\dot{\theta} = \frac{x}{2} \sin\theta + \left(\frac{n- 1}{r} -\frac{r}{2}\right)\cos\theta,
\end{cases}
\eeq
where $\theta $ is the angle that $\dot{\gamma}$ makes with the $x$-axis, and where ``$\cdot$'' denotes derivation in the arc length parameter. We will use this notation throughout the article.

Thus, for a hypersurface of revolution generated by a graph $u\in C^2(I)$ over an interval on the $x_1$-axis, $u:I\to\Reals^+$, the self-shrinker equation is
\beq
H(u(x))=\frac{\langle\vec{X},\vec{\nu}\rangle}{2}=\frac{1}{2}\frac{u(x)-xu'(x)}{(1+(u')^2)^{\frac{1}{2}}},
\eeq
which is equivalent to the following ODE
\beq\label{SSODE}
u''(x)=\Big[\frac{xu'(x)-u(x)}{2}+\frac{n-1}{u(x)}\Big]\big(1+(u'(x))^2\big),
\eeq
which is a cubic-derivative quasi-linear elliptic second-order equation of the following form
\[
u'' - xp(x, u'(x))u' + p(x, u'(x)) u = g(u(x), u'(x)),
\]
for appropriately defined functions $p$ and $g$. For the graph of a function $f$ over the $r$-axis, the equation becomes

\begin{equation} \label{r_graph}
 f''(r) = \left\{ \left(\frac{r}{2} - \frac{n-1}{r}\right) f'(r) - \frac{f(r)}{2} \right\} \left(1 + (f'(r))^2 \right).
\end{equation}

For such equations, containing a non-linearity of the form $(u')^3$, the general existence results by Nagumo and others do not apply (see for instance the survey \cite{CH}) and we will be developing an approach from scratch. Furthermore, note that the sign of the terms in (\ref{SSODE}) are such that one does not have the best possible maximum and convexity principles, but instead an oscillating behavior (e.g. as is the case for the linear equations $u''+bu'+cu=0$ when $c > 0$ is positive), contrasting for example the situation one would have had for self-expanders. Much of the intricacy concerning Equation (\ref{SSODE}) stems from this fact, and also from the lack of known exact symmetries.

The reader will notice that, in the proof of Theorem \ref{ThmUni} (e.g. in Proposition \ref{embedded_characterization}), solutions to Equation (\ref{ode_system}) are translated to get contradictions via a maximum principle, as in the method of moving planes. However, here translation is not an isometry for the geometric problem in (\ref{SSEq}), and correspondingly is not an invariance for (\ref{ode_system}). In certain situations, depending on the relative position of solutions and signed direction of translation, it is even ``better'' than an exact symmetry, a key special feature which we exploit repeatedly in our maximum principle arguments.
 
Few known examples of complete embedded hypersurfaces in $\Reals^n$ satisfying the self-shrinker equation, and indeed several non-existence results are known. In the paper \cite{Hu1}, Huisken showed that the only positive mean curvature $H\geq0$ rotationally symmetric hypersurface $\Sigma^n$, defined by revolution of an entire graph over the $x_1$-axis is the cylinder.

However, without the curvature assumption $H\geq0$ it is still an open question as to whether there could exist for example non-standard embedded self-shrinking spheres, planes or cylinders. Note in this connection that Angenent in \cite{Ang} gave numerical evidence for many non-round immersed spheres with a rotational symmetry axis. Our Theorem \ref{ThmUni} answers the question negatively under the assumption of embeddedness as well as a rotational symmetry axis of the hypersurface. Thus there are no analogues of the members of the rotationally symmetric Delaunay unduloid family of embedded, complete, singly periodic constant mean curvature surfaces that in the $H\equiv1$ case interpolates between the round cylinder and a string of round spheres touching at antipodal points (see \cite{De} or \cite{KK}). However, as Theorem \ref{ThmConicalEndsWordyVersion} demonstrates there does exist a family of self-shrinkers (with boundary) interpolating between the flat plane and round cylinder orthogonal to the plane.

Notice that the existence of the ``trumpet'' family of self-shrinkers as in Theorem \ref{ThmConicalEndsWordyVersion} (and its precise version in the below Theorem \ref{ThmConicalEnds}) along with the maximum principle for Equation (\ref{SSEq}) places certain crude a priori restrictions on what the non-compact ends of a general self-shrinker can be. Without investigating such issues further at present, let us remind the reader that this is related to announced work by Tom Ilmanen \cite{Il} stating that self-shrinkers have ends that are (in Hausdorff sense) asymptotically conical.

As mentioned in the introduction, there has been recent interest in applying the desingularizing methods of Kapouleas to the construction of complete embedded self-shrinkers with genus. Apart from its use in the proof of Theorem \ref{ThmUni}, the perspective of such constructions is one of the main interests of Theorem \ref{ThmConicalEndsWordyVersion}.

We note that simultaneously with our work for this paper, the recent monograph Giga-Giga-Saal \cite{GGS10} was also concerned with different proofs of well-known weaker versions of the uniqueness of self-shrinkers given by entire cylindrical graphs, which dates back to Huisken \cite{Hu1} (see also Soner-Souganidis \cite{SS93} and Altschuler-Angenent-Giga \cite{AAG95}), a special case of the present paper. Note also that the below Lemma \ref{BlowUpLemma} alone removes the assumption of $H\geq 0$ from all such results, see Corollary \ref{ThmUisken}.

\section{Proof of Theorem \ref{ThmConicalEndsWordyVersion}: An integral identity for graphs}
The version of Theorem \ref{ThmConicalEndsWordyVersion} we will prove is more precisely stated as follows:

\begin{theorem}[:= Theorem 1']\label{ThmConicalEnds}
Let $n\geq2$. For each fixed ray from the origin,
\[
r_\sigma(x)=\sigma x,\quad r_\sigma:(0,\infty)\to\Reals^+,\quad\sigma>0,
\]
there exists a unique smooth graphical solution $u_\sigma:[0,\infty)\to\Reals^+$, of (\ref{SSODE}), asymptotic to $r_\sigma$.

Also, for $d>0$, any solution $u:(d, \infty) \rightarrow \mathbb{R}^+$ to (\ref{SSODE}) is either the cylinder $u\equiv\sqrt{2(n-1)}$, or is one of the $u_\sigma$ for some $\sigma=\sigma(u)>0$.

Furthermore, the following properties hold for $u_\sigma$ when $\sigma>0$:
\begin{itemize}
\item[(i)] $u_\sigma > r_\sigma$, and $u_\sigma(0)<\sqrt{2(n-1)}$,
\item[(ii)] $|u_\sigma(x)-\sigma x|=O(\tfrac{1}{x})$, and $|u_\sigma'(x)-\sigma|=O(\tfrac{1}{x^2})$ as $x\to+\infty$,
\item[(iii)] $\Sigma_\sigma$ generated by $u_\sigma$ has mean curvature $H(\Sigma_\sigma)>0$,
\item[(iv)]  $u_\sigma$ is strictly convex, and $0<u_\sigma'<\sigma$ holds on $[0, \infty)$,
\item[(v)]  $\gamma_\sigma$, the maximal geodesic containing the graph of $u_\sigma$, is not embedded.
\end{itemize}
\end{theorem}
This immediately gives the following corollary, where as an improvement over \cite{Hu1} (where $H\geq 0$ was required) we do not need any curvature assumption.

\begin{corollary}\label{ThmUisken}
Let $\Sigma^n$ be a smooth self-shrinking hypersurface of revolution, which is generated by rotating an entire graph around the $x_1$-axis. Then $\Sigma^n$ is the round cylinder $\Reals\times S^{n-1}$ of radius $\sqrt{2(n-1)}$ in $\Reals^{n+1}$.
\end{corollary}
\begin{proof}[Proof of Corollary]
Any entire graph is a graph over the right half axis. Theorem \ref{ThmConicalEnds} characterizes all such graphs, and in particular says that none are embedded excepting the cylinder solution. 
\end{proof}

Note that our Theorem \ref{ThmConicalEnds} amounts to the following interesting geometric fact, which we get since instead of $(n-1)$ we may take an arbitrary number $\alpha>0$ everywhere in our proofs.
\begin{corollary}
In the (non-complete) generalized Gaussian upper half-plane
\[
\mathbb{G}_\alpha=(\Reals_x\times\Reals_r^+,g_{ij}=r^{2\alpha} e^{-\frac{x^2+r^2}{4}}\delta_{ij}),
\]
for any $\alpha\geq 0$, there exists for each $\sigma>0$ a unique geodesic ray $u_\sigma$, with the properties in Theorem \ref{ThmConicalEnds}. Note that for the usual Gaussian metric (where $\alpha=0$), we have $u_\sigma\equiv r_\sigma$, i.e. straight lines through the origin.
\end{corollary}

We will need the following lemma, which observes sufficient conditions for solutions to become non-graphical.

\begin{lemma}\label{BlowUpLemma}
If $x_0\in(0,\infty)$, and $(x_0,x_\infty)$ is a maximally extended solution to the initial value problem
\beq\label{SSODESys}
\begin{cases}
&u''=\Big[\frac{x}{2}u'+\frac{n-1}{u}-\frac{u}{2}\Big]\big(1+(u')^2\big),\\
&u(x_0)=\sigma x_0,\\
&u'(x_0)\geq\sigma,
\end{cases}
\eeq
where $\sigma>0$. Then $x_\infty<\infty$, and if $u(x_0)\geq\sqrt{2(n - 1)}$, then $x_\infty\leq(1+\frac{1}{n-1})x_0$. Geometrically these initial conditions mean that $H(u(x_0))\leq0$ at the point on the hypersurface $\Sigma$.
\end{lemma}

\begin{proof}[Proof of Lemma \ref{BlowUpLemma}]
Defining $\Psi(x):=xu'(x)-u(x)$, we note that the initial conditions are equivalent to $\Psi(x_0)\geq 0$. Since
\begin{equation} \label{SolComparison}
\Psi'=\Big(\frac{\Psi}{2}+\frac{n-1}{u}\Big)(1+(u')^2)>\frac{\Psi}{2}\geq0,
\end{equation}
we see that $\Psi'>0$ and $\Psi\geq0$, so $u'(x)\geq u(x)/x>0$, for $x\geq x_0$. Thus in particular there always exists $x_0'\geq x_0$ such that $u(x_0')\geq\sqrt{2(n - 1)}$ and $\Psi(x_0')\geq 0$, and we can without loss of generality assume $u(x_0)\geq\sqrt{2(n - 1)}$.

If we define for $u$ the quantity
\[
\Phi(x):=\frac{x}{2}u'+\frac{n-1}{u}-\frac{u}{2},
\]
then by (\ref{SSODESys}) we have $\Phi(x_0)\geq\frac{n-1}{\sigma x_0}$. We claim that in fact $\Phi(x)\geq\frac{n-1}{\sigma x_0}$ for all $x\geq x_0$. Namely assuming this holds up to $x$ we have for $x\geq x_0$,
\begin{align*}
\frac{d}{dx}\Big(\frac{x}{2}u'+\frac{n-1}{u}-\frac{u}{2}\Big)&=\frac{x}{2}u''-(n-1)\frac{u'}{u^2}=\frac{x}{2}\Big(1+(u')^2\Big)\Phi-(n-1)\frac{u'}{u^2}\\
&\geq \frac{x}{2}\frac{n-1}{\sigma x_0}\Big(1+(u')^2\Big)-(n-1)\frac{u'}{u^2}\geq \frac{n-1}{2\sigma}>0,
\end{align*}
assuming that both $u(x)\geq\sqrt{2(n - 1)}$ and $u'(x)\geq\sigma$. In particular $u''\geq 0$, and hence the set of conditions
\beq
\begin{cases}
&\Phi(x)\geq\frac{n-1}{\sigma x_0},\\
&u'(x)\geq\sigma,\\
&u(x)\geq\sqrt{2(n - 1)},
\end{cases}
\eeq
are simultaneously preserved by the self-shrinker ODE in (\ref{SSODE}) as $x\geq x_0$ increases.

Using $\Phi(x)\geq\frac{n-1}{\sigma x_0}$, we get that $u''\geq\frac{n-1}{\sigma x_0}(1+(u')^2)$, for $x\geq x_0$, and integrating this inequality gives
\[
u'(x)\geq\tan\Big[(n-1)\frac{x-x_0}{\sigma x_0}+\arctan\sigma\Big],
\]
which finally leads to $x_\infty<\frac{\sigma x_0}{n-1}\Big(\tfrac{\pi}{2}-\arctan\sigma\Big)\leq\frac{n}{n-1}x_0 + x_0 <\infty$.
\end{proof}

\begin{remark}\label{OtherUiskenProof}
Incidentally Lemma \ref{BlowUpLemma} also removes the $H\geq 0$ assumption, yielding a different proof of Corollary \ref{ThmUisken}.
\end{remark}

\begin{lemma}[Integral identity]\label{MagicLemma}
Any solution $u: (d, \infty) \rightarrow \mathbb{R}^+$ to (\ref{SSODE}), where $d\geq0$, satisfies for some $\sigma=\sigma(u)\geq0$ the identity 
\beq\label{MagicFormulaLong}
u(x)=2(n-1) x\int_x^\infty \frac{1}{t^2} \bigg\{\int_t^\infty \frac{s}{2}\frac{1 + (u'(s))^2}{u(s)} e^{- \frac{1}{2} \int_t^s  z\left( 1 + (u'(z))^2 \right) dz}ds \bigg\}dt+\sigma x,
\eeq
when $x\in (d,\infty)$.
\end{lemma}
\begin{proof}
Suppose first that we are given a solution $u: (d, a) \rightarrow \infty$ over an interval $(d, a)$. We can regard the solution $u$ as solving an inhomogeneous linear equation determined by freezing the coefficients at $u$,
\beq\label{FrozenODE}
u''-\big(1+(\varphi')^2\big)\frac{x}{2}u'+\big(1+(\varphi')^2\big)\frac{u}{2}=(n-1)\frac{\big(1+(\varphi')^2\big)}{\varphi},
\eeq
where we have set $u = \varphi$. We can solve the resulting linear equation with variable coefficients, for $x\in (d,a)$, by making the observation that a pair of spanning solutions of the homogeneous linear equation are
\beq
u_1(x)=x,\quad\textrm{and}\quad u_2(x)=x\int_x^a \frac{e^{-\frac{1}{2}\int_s^a z (1 + (\varphi')^2) dz}}{s^2}ds,
\eeq
Then computing the Wronskian $W(s)=e^{- \frac{1}{2} \int_s^a z (1 + (\varphi'(z))^2)dz}$, and matching the initial conditions gives
\begin{eqnarray} \notag
u (x) & = &\frac{u(a)}{a}x+(u(a)-u'(a)a)x\int_x^a \frac{e^{-\frac{1}{2}\int_s^a z (1 + (u')^2) dz}}{s^2}ds\\ \label{GeneralSol} 
& + & (n-1)x\int_x^a \frac{1}{t^2} \bigg\{\int_t^a \frac{ s \left( 1 + (u'(s))^2 \right)}{u(s)} e^{- \frac{1}{2} \int_t^s  z\left( 1 + (u'(z))^2 \right) dz}ds \bigg\}dt.
\end{eqnarray} 

To complete the proof of (\ref{MagicFormulaLong}), we will show that for some limit $\sigma\geq0$,
\beq
\frac{u(a)}{a} \rightarrow \sigma, \quad\textrm{for}\quad a\to\infty
\eeq
\beq
(u(a)-u'(a)a)x\int_x^a \frac{e^{-\frac{1}{2}\int_s^a z (1 + u'^2) dz}}{s^2}ds \rightarrow 0, \quad\textrm{for}\quad a\to\infty.
\eeq

Recall that by Lemma \ref{BlowUpLemma}, for any solution $u: (d,\infty) \rightarrow \mathbb{R}^+$ the quantity $\Psi(x) =  xu'(x) - u(x)$ is pointwise negative. Thus the ratio $\frac{u(a)}{a}$ is monotonically decreasing in $a$, and hence converges to some limit $\sigma \geq 0$. The negativity of $\Psi$ also implies that 
\beq\label{LowUBound}
u(x) \geq (n-1)x\int_x^a \frac{1}{t^2} \bigg\{\int_t^a \frac{ s \left( 1 + (u'(s))^2 \right)}{u(s)} e^{- \frac{1}{2} \int_t^s  z\left( 1 + (u'(z))^2 \right) dz}ds \bigg\}dt  + \sigma x.
\eeq

By this it follows that there exists a sequence $\{a_k\}$ increasing to infinity such that $u(a_k) \geq \sqrt{2(n-1)}$. Namely, otherwise one would have that $u(x) < \sqrt{2(n-1)}$ for large enough $x$. With (\ref{LowUBound}) we get for such large $x$ that
\[
u(x)\geq \frac{2(n-1)}{\sqrt{2(n-1)}} - R(a)\to \sqrt{2(n-1)},\quad\textrm{for}\quad a\to\infty,
\]
where $R(a)$ is an explicit error term, yielding the contradiction $u(x)\geq\sqrt{2(n-1)}$.

Moreover, we can modify the sequence $\{a_k\}$ to satisfy in addition $u'(a_k) \geq 0$. This is easily seen to follow from Equation (\ref{SSODE}) and the mean value theorem, using that $u(a_k)\geq\sqrt{2(n-1)}$ on the original sequence. Thus we have
\[
0 < u(a_k) - u'(a_k)a_k < \sqrt{2(n-1)},
\]
so that this term is bounded, and since
\[
\int_x^{a_k} \frac{e^{-\frac{1}{2} \int_s^{a_k} z(1 + (u
'(z))^2) dz}}{s^2} ds \leq \frac{e^{-\frac{a_k^2}{4}}}{x^2}\int_x^{a_k}e^{\frac{s^2}{4}}ds\to 0,\quad\textrm{for}\quad a_k\to\infty.
\]
we see that inserting the sequence $a_k \rightarrow \infty$ in (\ref{GeneralSol}) leads to (\ref{MagicFormulaLong}).
\end{proof}

As an immediate consequence of Lemma \ref{MagicLemma} we see that $u_\sigma(x) > \sigma x$, i.e. $u_\sigma>r_\sigma$, which leads to the following $L^\infty$-estimates.

\begin{lemma}\label{SupEstimates}
Let $u:(d,\infty)\to\Reals^+$ be as in Lemma \ref{MagicLemma}, with $\sigma>0$. Then
\begin{align}
&\sup_{s\in(x,\infty)}|u(s)-\sigma s|\leq \frac{2(n-1)}{\sigma x},\label{FuncBound}\\
&\sup_{s\in(x,\infty)}|u'(s)-\sigma|\leq\frac{2(n-1)}{\sigma x^2}\label{GradBound},
\end{align}
for $x\in(d,\infty)$. In particular $u$ extends to $u_\sigma:(0,\infty)\to\Reals^+$.
\end{lemma}
\begin{proof} We can estimate using $u >r_\sigma$,
\begin{align*}
|u(x) - \sigma x| \leq &\frac{2(n-1)}{\sigma} \int_x^\infty \frac{1}{t^2}\bigg\{\int_t^\infty \frac{s}{2}(1+(u'(s))^2) e^{-\int_t^s \frac{z}{2}(1+(u'(z))^2)dz} ds\bigg\} dt\\
\leq &\frac{2(n-1)}{\sigma x}
\end{align*}
and by similar reasoning obtain the estimate for $u'$.
\end{proof}
We can thus assume without loss of generality that $d=0$. 

To prove existence of a solution $u_\sigma$ for any $\sigma$,  we find it illustrative to construct each solution $u_\sigma$ as a limit of approximating solutions. More specifically, fixing a $\sigma > 0$, we solve the initial value problem
\beq\label{IVP}
\begin{cases}
&u''=\Big[\frac{x}{2}u'+\frac{n-1}{u}-\frac{u}{2}\Big]\big(1+(u')^2\big),\\
&u(a)=a\sigma,\\
&u'(a)=\sigma.
\end{cases}
\eeq
for $a$ positive. Denoting the solution $u_{\sigma, a}$, one derives the analogous identity 
\beq\label{MagicFormula}
u_{\sigma, a} (x)=(n-1)x\int_x^a \frac{1}{t^2} \bigg\{\int_t^a \frac{\left( 1 + u_{\sigma, a}'(s)^2 \right)}{u_{\sigma, a}(s)} e^{- \frac{1}{2} \int_t^s  z\left( 1 + u_{\sigma, a}'(z)^2 \right) dz}ds \bigg\}dt + \sigma x,
\eeq
for $x < a$. The lack of terms in this expression corresponding to the homogeneous equation is a special property of the initial conditions. One derives uniform estimates analogous to (\ref{FuncBound})--(\ref{GradBound}) for the solutions,
\begin{align}
&\sup_{s\in(x,a)}|u_{\sigma,a}(s)-\sigma s|\leq \frac{2(n-1)}{\sigma x},\\
&\sup_{s\in(x,a)}|u_{\sigma,a}'(s)-\sigma|\leq\frac{2(n-1)}{\sigma x^2},
\end{align}
for any $x<a$. This gives that each solution $u_{\sigma, a}$ extends to $(0, a)$, and by compactness that the family $\left\{u_a \right\}_{a> 0}$ converges to a limiting solution $u_\sigma$ on $(0, \infty)$, uniformly in the $C^2$-norm on compact sub-intervals.

Note however, that each approximate solution is really only approximate: Lemma \ref{BlowUpLemma} implies that they do not remain graphical for values of $x$ much larger than $a$, but bend upwards with $u'_{\sigma,a}(x)\to\infty$ as $x\to x_\infty<\infty$.

We next prove that any solution $u_\sigma:(0,\infty)\to\Reals^+$ asymptotic to the ray $r_\sigma$ is unique. Recall that we have shown that, given a $\sigma$, any solution $u_\sigma$ must satisfy
\begin{equation}\label{SolProperties}
u_\sigma(x) > \sigma x
\end{equation}
as well as the $L^\infty$-estimates in (\ref{FuncBound})--(\ref{GradBound}). Consider for $b>0$ the Banach space
\[
C^1_0([b,\infty))=\left\{v:[b,\infty)\to\Reals\mid v,v'\in C_0([b,\infty))\right\}
\]
of continuously differentiable functions $v$ such that $|v(x)|\to 0$ and $|v'(x)|\to 0$ as $x\to+\infty$, endowed with the uniform $C^1$-norm $\|v\|_{C^1}=\|v\|_{\infty}+\|Dv\|_{\infty}$, where the supremum is taken over $[b,\infty)$.

Also, for $b,\sigma>0$ we can for example consider the open subsets
\[
Y_{\sigma, b}:=\left\{v\in C^1_0([b,\infty))\:\Big|\:v(x) > 0, \quad |v'(x)|<\frac{4(n-1)}{\sigma x^2}\right\},
\]
so that by the estimates (\ref{FuncBound})--(\ref{GradBound}) the solutions to (\ref{SSODE}) - $\sigma x$ are in $Y_{\sigma, b}$.

Then we will show that the non-linear mapping $T_\sigma$ on $Y_{\sigma, b}$ given by
\[
[T_\sigma v](x) = 2(n-1)x \int_x^\infty \frac{1}{t^2}\bigg\{\int_t^\infty \frac{1+(v'(s)+\sigma)^2}{v(s)+\sigma s} \frac{s}{2}e^{-\int_t^s \frac{z}{2}(1+(v'(z)+\sigma)^2)dz} ds\bigg\} dt
\]
is a contraction, if $b=b(n,\sigma)$ is chosen large enough. Note that if $u$ is a solution to the equation (\ref{SSODE}), then by the integral identity in Lemma \ref{MagicLemma}
\[
[\tilde{T}_\sigma u](x):=[T_\sigma (s\mapsto u(s)-\sigma s)](x)+\sigma x=u(x),
\]
and conversely, so that $v(x)+\sigma x$ solves equation (\ref{SSODE}) if and only if $T_\sigma v=v$.

In fact $T_{\sigma,b}$ is well-defined, and we get the mapping property
\[
T_\sigma: Y_{\sigma, b} \rightarrow Y_{\sigma, b},
\]
as follows similarly to the proofs of the estimates in Lemma \ref{SupEstimates} and of the properties (\ref{SolProperties}), using the integral identity in Lemma  \ref{MagicLemma}.
\begin{proposition}
There exists $b_0=b_0(n,\sigma)$ such that $T_\sigma$ is a contraction for the norm $\|\cdot\|_{C^1}$ on the set of functions $Y_{\sigma, b}$ for $b\geq b_0$.
\end{proposition}

\begin{proof}
For two functions $v_1, v_2\in Y_{\sigma,b}$ we may write $u_i(x)=v_i(x)+\sigma x$ and get for $\tilde{T}_\sigma$ the expression:
\begin{eqnarray*}
\tilde{T}_\sigma u_2 - \tilde{T}_\sigma u_1 & = & 2(n-1)x\int_x^\infty \frac{1}{t^2} \int_t^\infty \left( \frac{1}{u_2} - \frac{1}{u_1} \right) \frac{s}{2}(1+(u_2')^2) e^{-\int_t^s \frac{1}{2}z(1+(u_2'(z))^2)} \\
& + & 2(n-1)x\int_x^\infty \frac{1}{t^2} \int_t^\infty \frac{\frac{s}{2}(1+(u_2')^2)}{u_1}\left(e^{-\int_t^s \frac{z}{2}(1+(u_2')^2)} - e^{-\int_t^s\frac{z}{2}(1+(u_1')^2)} \right) \\
& + & 2(n-1)x\int_x^\infty \frac{1}{t^2} \int_t^\infty \frac{1}{u_1} \frac{s}{2}\left((u_2')^2 - (u_1')^2\right) e^{-\int_t^s \frac{z}{2}(1+(u_1')^2)}\\
& =: & A + B + C.
\end{eqnarray*}
We estimate the term $A$ by
\[
A \leq 2(n-1)\frac {|| u_2 - u_1 ||_\infty}{\sigma^2x^2}.
\]
The term $C$ may be estimated by 
\begin{align*}
C &\leq || u_2' - u_1'||_\infty ||u_2' + u_1'||_\infty \frac{n-1}{\sigma}\int_x^\infty \frac{1}{t^2}\bigg\{\int_t^\infty s e^{\frac{1}{4} (t^2 - s^2)} ds\bigg\}dt\\
& = ||u_2'-u_1'||_\infty ||u_2' + u_1'||_\infty \frac{ 2(n-1)}{ \sigma x}.
\end{align*}
To estimate the term $B$, we note that, for real numbers $x, y \leq c$, one has $|e^y - e^x| \leq e^c |y - x|$ so that we may estimate term $B$ as follows:
\begin{eqnarray*} \notag
B & \leq  &||u_2' + u_1' ||_\infty ||u_2' - u_1' ||_\infty || 1 + (u_2')^2||_\infty \frac{(n-1)x}{\sigma} \int_x^\infty \frac{1}{t^2}\bigg\{\int_t^\infty \frac{1}{4}(s^2 - t^2)e^{\frac{1}{4}(t^2 - s^2)}  ds\bigg\}dt \\ \notag
& \leq & ||u_2' + u_1' ||_\infty ||u_2' - u_1' ||_\infty || 1 + (u_2')^2||_\infty \frac{(n-1)x}{2 \sigma} \int_x^\infty \frac{1}{t^3}\bigg\{\int_t^\infty \frac{1}{2}(s^2 - t^2)e^{\frac{1}{4}(t^2 - s^2)} s ds\bigg\}dt \\ \notag
& = & ||u_2' + u_1' ||_\infty ||u_2' - u_1' ||_\infty || 1 + (u_2')^2||_\infty \frac{(n-1)x}{2 \sigma} \int_x^\infty \frac{1}{t^3}\bigg\{\int_0^\infty \tau e^{- \tau} d\tau\bigg\} dt \\
& \leq & ||u_2' + u_1' ||_\infty ||u_2' - u_1' ||_\infty || 1 + (u_2')^2||_\infty \frac{(n-1)}{2 \sigma x}.
\end{eqnarray*}

Also, we may write
\[
(\tilde{T}_\sigma u )' = \frac{\tilde{T}_\sigma(u)}{x} - \frac{2(n-1)}{x} \int_x^\infty \frac{\frac{s}{2}\left( 1 + (u'(s))^2 \right)}{u(s)} e^{-\int_x^s  \frac{z}{2}\left( 1 + (u'(z))^2 \right) dz}ds+\sigma,
\]
and from this representation formula similarly get, for $p_i(s)=\frac{s}{2}(1+(u_i'(s))^2)$:
\begin{eqnarray} \notag
(\tilde{T}_\sigma u_2)' - (\tilde{T}_\sigma u_1)' & =  &\frac{1}{x}\left(T_\sigma u_2 - T_\sigma u_1 \right)  \\ \notag
 & - & \frac{2(n - 1)}{x} \int_x^\infty \left(\frac{1}{u_2(s)} - \frac{1}{u_1(s)}\right)p_2(s) e^{-\int_x^s p_2(z) dz} \\ \notag
 & - & \frac{2(n - 1)}{x} \int_x^\infty \frac{p_2(s)}{u_1(s)}\left(e^{-\int_x^s p_2(z) dz} - e^{- \int_x^s p_1(z) dz} \right) \\ \notag
 & - & \frac{2(n - 1)}{x} \int_x^\infty \frac{1}{u_1(s)} \left(p_2(s) - p_1(s) \right)e^{- \int_x^s p_1(z)dz} \\ \notag
 & = & \frac{1}{x} \left(T_\sigma u_2 - T_\sigma u_1 \right) - A' - B' - C'.
\end{eqnarray}
Then the terms $A'$, $B'$, and $C'$ may be treated similarly to the terms $A$, $B$, and $C$ before.

Thus we see, going back to $v_i$ and to $T_\sigma$, that:
\beq\label{Contraction}
||T_\sigma v_2 - T_\sigma v_1||_{C^1} < \tau ||v_2 - v_1||_{C^1},
\eeq
for some $0<\tau<1$, if we choose $b_0$ large enough, and with the $C^1$-norm taken over $(b_0,\infty)$. Thus $T_\sigma: Y_{\sigma,b_0} \rightarrow Y_{\sigma,b_0}$ is a contraction.

Note also that a family version of the Proposition follows, that is if $0<\sigma_i<\infty$ are given, then $b_0$ and $\tau$ can be chosen so that (\ref{Contraction}) holds uniformly for $\sigma\in[\sigma_1,\sigma_2]$
\end{proof}

Applying the Proposition shows the claimed uniqueness for graphs over half-lines satisfying Equation (\ref{SSODE}). Namely, let two solutions $u_1$ and $u_2$ to the equation for the same $\sigma$-value be given. Then for $b_0$ chosen large enough we have $u_1,u_2\in Y_{\sigma,b_0}$ and the result follows.

\begin{remark}
Since the map $T_\sigma$ is a contraction for large enough $x$-values, one can also prove the existence part (for large $x$) of Theorem \ref{ThmConicalEndsWordyVersion} using a fixed point principle.
\end{remark}
The graphs of the  functions $u_\sigma$ constructed above are eventually graphical over the $r$-axis as well (since they are eventually increasing), and are given by functions $f_{1/\sigma}: [r_{1 / \sigma}, \infty] \rightarrow \mathbb{R}$ on some maximal domain $(r_{1 / \sigma},  \infty)$. The functions $f_{1/\sigma}$ then satisfy equation (\ref{r_graph}), and an analysis similar to that in the proof of Lemma \ref{MagicLemma} gives that the $f_{1/\sigma}$ satisfy the identity
\beq\label{r_graph_identity}
S_{\sigma}f_{1/\sigma}=f_{1/\sigma},
\eeq
where the map $S_{\sigma}$ given by 
\begin{equation}
f \mapsto  \frac{r}{\sigma} - (n - 1)r \int_r^\infty \frac{1}{t^2}\int_t^\infty f'(s) \left(1 + (f'(s))^2 \right) e^{-\int_t^s\frac{z}{2}(1 + (f'(z))^2)dz}ds dt,
\end{equation}
which is then similarly shown to be a contraction mapping. The fixed points of the maps $S_{1 / \sigma} $ and $T_\sigma$ then determine a complete geodesic $\gamma_\sigma$ in the upper half plane.  We now show that the $\gamma_\sigma$ depend smoothly on the parameter $\sigma$ in the $C^k$ topology. For this, we will need the following general fact, proved in the Appendix.

\begin{lemma}\label{FixSmooth}
Let $\Phi_\sigma: Y \rightarrow Y$ be a smooth one parameter family of smooth contraction mappings on a fixed open subset $Y$ of a Banach space $X$. Then the fixed points $x_\sigma$ (assumed to exist) are smooth functions of $\sigma$.
\end{lemma}

Thus, the solutions $\gamma_\sigma$ depend smoothly in $C^1$ on the parameter $\sigma$. However, the geodesic equation gives that the dependence is smooth in $C^k$ for any $k$.

\begin{lemma}\label{f_sigma_smooth}
The map $\sigma \mapsto f_{\sigma}: \mathbb{R}^+ \rightarrow C^k$ is smooth for any $k$.
\end{lemma}
\begin{proof}
The equation 
\begin{equation}
f''_{\sigma}(r) = \left\{ \left(\frac{r}{2} - \frac{n - 1}{r} \right)f_{\sigma}'(r) - \frac{f_{\sigma}(r)}{2} \right\}(1 + (f_{\sigma}')^2),
\end{equation}
immediately gives that the second derivative $f''_{\sigma}$ is differentiable in $\sigma$. Differentiating (\ref{r_graph}) in $r$ then gives that all higher derivatives $f_{\sigma}^{(k)}$ are differentiable in $\sigma$ as well.
\end{proof}

By Lemma \ref{f_sigma_smooth}, the function $F(\sigma,r)$ given by
\begin{equation}
F(\sigma, r) = f_{\sigma}(r) 
\end{equation}
is smooth on its domain of definition. Note that, as $\sigma \rightarrow \infty$ the functions $f_{1 /\sigma}$ converge to the function $ f_0 (r) \equiv 0$  uniformly in $C^k$ on compact subsets of $(0, \infty)$ for any $k$. Thus, defining $\hat{\sigma} = 1 / \sigma$, it follows that the function $g(r) = \frac{df}{d \hat{\sigma}}|_{\hat{\sigma} = 0}(r)$ is defined on $(0, \infty)$ and satisfies the linearized equation
\begin{equation} \label{linearized_equation}
g''(r) = \left(\frac{r}{2} - \frac{n - 1}{r}\right) g'(r) - \frac{g(r)}{2} .
\end{equation}

To analyze solutions of the linearized equation, we again prove an integral identity.

\begin{lemma} The solution to the linearized equation $g$ above satisfies the identity
\begin{equation} \label{linearized_identity}
 g(r) = r - (n - 1)r\int_r^\infty \frac{1}{t^2} \int_t^\infty g'(s) e^{(t^2 - s^2)/4} ds dt.
\end{equation}
\end{lemma}
\begin{proof}
Differentiating identity (\ref{r_graph_identity}) (with $\hat{\sigma} = 1 / \sigma$), we obtain
\begin{equation}
f'_{\hat{\sigma}}(r) = f_{\hat{\sigma}}/r + \frac{n - 1}{r} \int_r^\infty f'_{\hat{\sigma}}(s)(1 + (f'_{\hat{\sigma}}(s))^2)e^{-\int_r^s\frac{z}{2}(1 + (f'_{\hat{\sigma}}(z))^2)dz} ds.
\end{equation}
Thus, for $r > 2(n - 1)$, we get
\begin{equation} \label{dominated}
f'_{\hat{\sigma}} (r) / \hat{\sigma} < (1 - 2(n - 1)/ r^2)^{-1}.
\end{equation}

Now, since
\begin{equation}
 f_{\hat{\sigma}}'(s) / \hat{\sigma} \left(1 + (f_{\hat{\sigma}}'(s))^2 \right) e^{-\int_t^s\frac{z}{2}(1 + (f_{\hat{\sigma}}'(z))^2)dz} \rightarrow g'(s) e^{(t^2 - s^2)/4},\quad\textrm{as}\quad \hat{\sigma} \rightarrow 0.
\end{equation}
the above estimate (\ref{dominated}) gives convergence of the equation (\ref{r_graph_identity}) divided by $\hat{\sigma} = 1 / \sigma$, as $\hat{\sigma} \rightarrow 0 $ to (\ref{linearized_identity}) by, for example, the dominated convergence theorem.
\end{proof} 

\begin{corollary}\label{pos_and_neg}
The solution $g$ to the linearized equation assumes both positive and negative values on $(0, \infty)$.
\end{corollary}
\begin{proof}
Assume first that $g > 0$ everywhere on $(0, \infty)$. Note that we must then also have $g' > 0$ everywhere. For suppose $g'(r_0) \leq 0$ at some $r_0$. Then appealing to equation (\ref{linearized_equation}), we see that $g'(r) < 0$ for all $r > r_0$. In particular, for $r > \sqrt{2(n - 1)}$, we get $g''(r) < 0$, which implies that the graph of $g$ will eventually intersect the $r$-axis, a contradiction.

Thus we have $g' > 0$ on $(0, \infty)$. However, the identity (\ref{linearized_identity}) then gives the contradiction 
\begin{equation}
\lim_{r \rightarrow 0 }g(r) = -(n - 1)\int_0^\infty g'(s) e^{- s^2/4} ds< 0.
\end{equation}
Since the equation (\ref{linearized_identity}) is linear homogeneous, it follows that $g<0$ cannot hold either.
\end{proof}

\begin{lemma} \label{positive_slope_almost_there}
The functions $u_\sigma$ have positive slope on $[0, \infty)$ for $\sigma > 0$ sufficiently large. 
\end{lemma}
\begin{proof}
As before, take $\hat{\sigma} = 1 / \sigma $. Then the graphs $f_{\hat{\sigma}}$ are defined on the maximal interval $(r_{\hat{\sigma}}, \infty)$ (that is, $\lim_{ r \rightarrow r_{\hat{\sigma}}^+} f'_{\hat{\sigma}}(r) \rightarrow \infty$). Note that $r_{\hat{\sigma}} \rightarrow 0 $ as $\hat{\sigma} \rightarrow 0$, since the graphs $f_{\hat{\sigma}}$ converge uniformly to $0$ in any $C^k$ on compact subsets of $(0,  \infty)$.

Now, let $r_0$ be a point such that $\frac{\partial f}{\partial\sigma}(r_0) = g(r_0) < 0$. Then, choosing $\hat{\sigma}$ sufficiently small so that $r_{\hat{\sigma}} < r_0$, we get that 
\begin{equation}
f_{\hat{\sigma}}(r_0) = g(r) \hat{\sigma} + O(\hat{\sigma}^2) < 0, 
\end{equation}
for $\hat{\sigma}$ small enough. Since each $f_{\hat{\sigma}}$ is eventually positive, we see that there is a largest point $m_{\hat{\sigma}}$ such that $f_{\hat{\sigma}}(m_{\hat{\sigma}}) = 0$. Thus $f'_{\hat{\sigma}}(m_{\hat{\sigma}})>0$. Then the graph of $f_{\hat{\sigma}\mid[m_{\hat{\sigma}}, \infty]}$ is also graphical over the $x$-axis, and defines the solutions $u_\sigma$ for $\hat{\sigma} = 1 / \sigma$. Thus, for $\sigma$ sufficiently large, the function $u_\sigma$ is increasing.
\end{proof}

As corollaries to the above, we now obtain the properties (i) and (iv) listed in Theorem \ref{ThmConicalEnds}.

\begin{proof}[Proof of Theorem \ref{ThmConicalEnds}(iv)]
Firstly we prove the second part of Theorem \ref{ThmConicalEnds}(iv), namely that the functions $u_\sigma: [0,\infty)\to\Reals^+$ are strictly increasing for any $\sigma>0$.

By Lemma \ref{positive_slope_almost_there} this is true for large enough $\sigma>0$. Assume there exists a $\sigma>0$, and hence a largest $\sigma_0>0$, such that this is not true. Then there is a point $x_0>0$ such that $u_{\sigma_0}'(x_0)=0$ and since $\sigma_0$ is the largest such, then by continuity of the solution in $\sigma$, we must have $u_{\sigma_0}(x_0)=\sqrt{2(n-1)}$ unless $x_0=0$ (since else by (\ref{SSODE}) there would be a point $x_0'\neq x_0$ such that $u_\sigma'(x_0')<0$ violating the maximality). Thus $u_{\sigma_0}$ is in that case is the cylinder, a contradiction. Since in the other case $u_{\sigma_0}'(0)=0$, we get by reflection a smooth, entire graphical surface of revolution with $H\geq0$ and thus by \cite{Hu1} we get that $u_{\sigma_0}$ is the cylinder, again a contradiction. Thus Lemma \ref{positive_slope_almost_there} extends to all $\sigma>0$.

As a corollary, we get the convexity in Theorem \ref{ThmConicalEnds}(iv), i.e. that $u_\sigma$ is strictly convex on $[0, \infty)$ for $\sigma > 0$. Namely, differentiating (\ref{MagicFormulaLong}) twice, we obtain
\begin{eqnarray*}
\frac{u_\sigma''}{1+(u_\sigma')^2}=(n-1)\bigg[\frac{1}{u_\sigma(x)}-\int_x^\infty\frac{\frac{s}{2}(1+(u_\sigma'(s))^2)}{u(s)}e^{-\int_x^s\frac{z}{2}(1+(u_\sigma'(z))^2)dz}ds\bigg],
\end{eqnarray*}
and hence $u_\sigma''>0$ on $[0, \infty)$, since $u_\sigma(s) > u_\sigma(x)$ for $s > x$. 
\end{proof}

We also get the second property in Theorem \ref{ThmConicalEnds}(i).
\begin{proof}[Proof of Theorem \ref{ThmConicalEnds}(i)]
Using the integral identity in Lemma \ref{MagicLemma} for $u_\sigma$ gives the following sharp bound on the value of $u(0)$, using l'H\^{o}pital's rule:
\begin{align*}
u(0)\leq\frac{2(n-1)}{u(0)}\int_0^\infty
\frac{s}{2}(1+(u'(s))^2)e^{-\int_0^s \frac{z}{2}(1+(u'(z))^2)dz}ds\leq\frac{2(n-1)}{u(0)},
\end{align*}
with sharp inequality unless $u\equiv u(0)$, so that
\[
u(0)\leq \sqrt{2(n-1)},
\]
with equality if and only if $u$ is the cylinder solution. 
\end{proof}

\section{Classification of self-shrinkers with rotational symmetry} \label{classification_section}

In this section we prove Theorem \ref{ThmUni}, which we restate here for the convenience of the reader in the context of geodesics in the upper half plane $(H^+, g_{\text{Ang}})$.

\begin{theorem} \label{classification_theorem}
Let $\gamma$ be a complete embedded geodesic for the metric $g_{\text{Ang}}$ in the upper half plane $H^+$. Then the following statement hold.
\begin{itemize}
\item[(1)] If $\gamma$ is closed, it is a curve that intersects the $r$-axis exactly twice.
\item[(2)] If $\gamma$ is not closed, it is either the $r$-axis, the line $r = \sqrt{2(n - 1)}$, or the sphere $x^2 + r^2 = 2n$.
\end{itemize}
\end{theorem}
\begin{corollary}[= Theorem \ref{ThmConicalEnds}(v)]
In particular this implies the remaining part (v) in Theorem \ref{ThmConicalEnds}, that the asymptotically conical ends are not parts of complete, embedded self-shrinkers.
\end{corollary}

To facilitate the discussion, we  say that a point in a smooth curve is  ``vertical'' if the tangent line at that point is parallel to $ e_r$, and ``horizontal'' if parallel to $e_x$, where $\{e_x, e_r\}$ is the unit basis corresponding to the Euclidean coordinates $(x, r)$ on $H^+$. By the first and second quadrants, we as usual mean the sets $\{(x, r)\mid x, r > 0  \}$ and $\{ (x, r)\mid x < 0, r > 0\}$ contained in $H^+$, respectively. For a smooth  curve $\gamma(t) = (x(t), r(t))$ parametrized by Euclidean arc length, we denote 
 \begin{equation} \notag
 \theta (t) = \arccos\dot{x} (t) = \arctan (\dot{r}(t) /\dot{x}(t)),
 \end{equation}
  and we say that $\gamma(t)$ is a solution to (\ref{ode_system}), if the triple $(x(t), r(t), \theta(t))$ solves (\ref{ode_system}). We occasionally refer to such curves $\gamma$ as ``geodesics'' for the metric $g_{\text{Ang}}$ in $H^+$, although this is a slight abuse of terminology as solutions to (\ref{ode_system}) are parametrized by Euclidean arc length, not arc length with respect to $g_{\text{Ang}}$. We will  make  frequent use of the following elementary observation.
 
  \begin{lemma} \label{higgins}
 Let  $\gamma(t) = (x(t), r(t))$ be a solution  to (\ref{ode_system}). Then the functions $x(t)$ and $r(t) - \sqrt{2(n - 1)}$ have neither positive minima nor negative maxima, and in particular these functions have different signs at successive critical points.
 \end{lemma}
 
 \begin{remark}
 We remind the reader that the reflection $(x,r)\mapsto(-x,r)$ is a symmetry of the equation, a fact that will be used often in the below.
 \end{remark}
 
 The following lemma is of fundamental importance for our proof.  Included in the statement of (2) is the (geometrically unsurprising) fact that geodesics for the metric $g_{\text{Ang}}$ that leave the upper half plane through the $x$-axis do so orthogonally. 
 
\begin{lemma} \label{maximal_graphs}
Let $\gamma: (a, b) \rightarrow H^+$ be a solution to (\ref{ode_system}), maximally extended as a graph over the $x$-axis. Then
\begin{itemize}
\item[(1)] There is $t \in (a, b)$ such that $x(t) = 0$.
\item[(2)] Assuming the existence and finiteness of the limit 
\begin{equation} \notag
x_b:=\lim_{t \rightarrow b^-} x(t)<\infty,
\end{equation}
the curve $\gamma$ extends smoothly to $(a, b]$, with $\gamma(b)$ a vertical point. If $r(b) = 0$, the curvature of $\gamma$ at $\gamma(b)$ (signed w.r.t. the orientation out of the half-plane) is $- \frac{x_b}{2 n}$.
\item[(3)] There is at least one horizontal point in $\gamma$. 
\end{itemize}
\end{lemma}

\begin{proof}
Assume the orientation of the curve is such that $\cos\theta=\dot{x}(t) > 0$ for $t \in (a, b)$. Set
\begin{equation}\label{Lambda_def}
 \Lambda(t) :  = x(t) \sin \theta (t) - r(t) \cos \theta (t) =   -   \langle\gamma(t),\nu(t)\rangle, 
\end{equation} 
where $\nu (t) = (- \sin \theta (t), \cos \theta(t))$ is the (leftward pointing w.r.t $\dot{\gamma}$) unit normal to $\gamma$. Then (\ref{ode_system}) becomes
\begin{equation}\label{ode_again}
\dot{\theta} = \frac{1}{2} \Lambda + \frac{n - 1}{r} \cos \theta,
\end{equation}
and $\Lambda$ satisfies the equation
\begin{equation} \label{Lambda_dot}
\dot{\Lambda}  = \frac{1}{2} \Lambda  \langle \gamma,  \dot{\gamma}  \rangle + \frac{n - 1}{r} \cos \theta \langle \gamma , \dot{\gamma} \rangle.
\end{equation}

We now investigate the oscillation behavior. Picking some $(x_0,u(x_0))$ on $\gamma$ and integrating by parts in (\ref{SSODE}) gives
\beq
\arctan u'\mid_{x_0}^x=xu(x)-x_0u(x_0)+\int_{x_0}^x\Big[\frac{n-1}{u(s)}-u(s)\Big]ds,
\eeq
so that if we assume a lower (resp. upper) bound on $r(t)=u(x)$, as $x\to x_b$, it leads to a uniform upper (resp. lower) bound on $u'(x)$. Therefore by the mean value theorem (recall that by Lemma \ref{higgins} successive points where $u'(x)=0$ must occur on either side of the line $r=\sqrt{2(n-1)}$), such points must either eventually stop occurring as $t\to b$, or the limit $r(t)\to\sqrt{2(n-1)}$ must hold. But if $u'(x)$ eventually has a fixed sign, then the limit $\lim_{t\to b} r(t)$ also exists.

Thus if we denote by $r^+_b  $ (resp. $r^-_b$) the $\limsup$ (resp. $\liminf$) of $r(t)$ as $t \rightarrow b$, then we have shown that either:
\begin{itemize}
\item[(i)] There is a limit: $\lim r_b = r^+_b = r_b^-$, or
\item[(ii)] Both $r^+_b = \infty$ and $r_b^- = 0$.
\end{itemize} 
But the second situation does not happen: Case (ii) implies that the straight line segment $\{(x_b, t): t > 0 \}$ satisfies $(\ref{ode_again})$, and thus we conclude $x_b=0$. But from (ii) we thus also obtained a positive solution $g(r)$ to the linearized equation at the $r$-axis (\ref{linearized_equation}), which gives a contradiction similarly to in Corollary \ref{pos_and_neg}.

Now, it is easy to see that the limit $r_b$ is finite: If $x_b \leq 0$, then assuming both $r(t) > \sqrt{2(n - 1)}$ and $\dot{r}(t) > 0$ then (\ref{ode_again}) gives that $\dot{\theta}(t) < 0 $, which immediately bounds $r_b$ away from $\infty$.

On the other hand, assuming still $r_b=+\infty$ but $x_b > 0$,  then (again by Lemma \ref{higgins}) eventually $\dot{r}(t)>0$ as $t\to b$, and hence eventually $\langle \gamma, \dot{\gamma}\rangle > 0$. There are also choices of $t_0 \in (a, b)$  arbitrarily close to $b$ with $\Lambda (t_0) > 0$, since else for some fixed $x^0<x_b$ we would have had $xu'(x)-u(x)<0$ for $x$ in an interval $(x^0, x_b)$, leading to the contradictory bound:
\beq
r_b=\lim_{x\to x_b}u(x)\leq x_b\frac{u(x^0)}{x^0}<\infty.
\eeq
Now, since $\langle \gamma(t), \dot{\gamma}(t) \rangle>0$, the property $\Lambda(t)>0$ is propagated on $t\in(t_0,b)$, by (\ref{Lambda_dot}). Dividing (\ref{Lambda_dot}) by $\Lambda$ and integrating over $(t_0, t)$ gives that
\begin{equation} \label{lambda_lower_bound}
\Lambda (t) > e^{|\gamma (t)|^2/4 - |\gamma (t_0))|^2/4} \Lambda (t_0).
\end{equation} 
However, combining (\ref{lambda_lower_bound}) with (\ref{ode_again}) and $|\gamma(t)|\to+\infty$ gives that 
\begin{equation}
\theta (t) - \theta ( t_0)  \rightarrow +\infty
\end{equation}
as $t \rightarrow b$, contradicting that $\gamma$ is graphical over the $x$-axis.

Thus the limit  $r_b$ exists and is a non-negative real number. If it is positive, then  $\gamma(t)$ remains in a relatively compact  subset of the upper half plane $H^+$ as $t \to b$.  Equation (\ref{ode_again}) then gives uniform $C^k$-bounds on $\gamma(t)$ and the desired smooth extension to a vertical endpoint, giving the conclusion $(2)$ in that case. Lemma \ref{higgins} then implies that $x_b > 0$. 

If on the other hand $r_b=0$,  then we claim that $\theta(t)$ decreases to $-\pi/2$ monotonically as $t$ increases to $b$. To see this, note first that $\theta(t)$ cannot remain bounded away from $-\pi/2$ as $t\to b$,  since otherwise  (\ref{ode_again}) and $r(t)\to 0$ give that
\beq
\dot{\theta}(t) \geq \frac{\delta}{r(t)} - 2x_b,\quad\textrm{where}\quad \delta:=\underset{t\nearrow b}{\inf}(\cos\theta(t)),
\eeq
for $t$ close enough to $b$. This, after using that $\dot{r}(t) \geq -1$ and integrating, gives
 \begin{equation} \notag
 \theta(t_2) - \theta (t_1) \geq \log \left(\delta\frac{r(t_1)}{r(t_2)}\right) - 2x_b(t_2 - t_1), 
 \end{equation} 
for any $t_2 > t_1$, and implicitly bounds $r(t)$ away from zero as $t\to b$, a contradiction.

In particular there must be points arbitrarily close to $b$ s.t. $\dot{\theta}>0$. Now, $r(t)\to 0$ and Lemma \ref{higgins} imply that $\dot{r}(t) < 0$ for $b - t$ sufficiently small, and differentiating (\ref{ode_again}) gives that 
\begin{equation} \label{theta_dotdot}
\ddot{\theta}(t) = - \frac{n - 1}{r^2}\dot{r}(t) \cos \theta (t),
\end{equation}
 at times $t$ for which $\dot{\theta} (t) = 0$, if there were any. Thus it follows that in fact $\dot{\theta}(t) < 0$ for all $b - t$ sufficiently small, and we have proved that $\theta(t) \searrow -\pi /2$ as $t\nearrow b$. 

Finally, applying l'H\^{o}pital's rule to (\ref{ode_again}) gives that 

\begin{equation} \notag
\lim_{t \rightarrow b^-} \dot{\theta} (t) = - \frac{x_b}{ 2 n}. 
\end{equation}
so that $\gamma(t)$  extends with two derivatives to $(a, b]$ with $x_b > 0$. The higher regularity then follows immediately, giving (2) also in this case.

In all cases, we see that $x_b > 0$. By symmetry, we get that also $x_a < 0$, which gives claim $(1)$.

To see (3), suppose first that $(a, b)$ is a bounded interval.  Note that (3) is clear if $\lim\theta(t)\in\{\pm\pi/2\}$ is different at the two endpoints. We may thus assume, with our chosen orientation, that $\lim_{t \rightarrow a+} \theta(t) = \lim_{t \rightarrow b-} \theta(t) = \pi/ 2$, and we argue by contradiction.

By (1), there is a $t_0 \in (a, b)$ so that $x(t_0) = 0$. If $r(t_0) > \sqrt{2(n -1)}$, then  (\ref{ode_again}) gives that $\dot{\theta}(t_0) < 0$. Differentiating (\ref{ode_again}) and evaluating at a point $t$ for which $\dot{\theta} (t) = 0$ gives that 
\begin{equation} \label{freakout}
\ddot{\theta} (t) = - \frac{n - 1}{r^2(t)} \dot{r}(t) \cos \theta (t) < 0, 
\end{equation}
so that $\theta(t)$ is bounded away from $\pi/2$ as $t \rightarrow b$,  a contradiction. If $r(t_0) < \sqrt{2(n - 1)}$, then (\ref{ode_again})  gives  $\theta(t_0) > 0$  and we apply a similar argument as before to contradict the assumption $\lim_{t \rightarrow a+} \theta(t) = \pi/2$. In the case of equality $r(t_0) = \sqrt{2(n - 1)}$, we have $\dot{\theta}(t_0) = 0$, and we refer to (\ref{freakout}) to obtain that $\dot{\theta}(t) < 0$ for $t > t_0$, from which as before we obtain a contradiction. 

If $b = \infty$ and $a$ is finite, then Theorem \ref{ThmConicalEnds} gives that $\gamma$ contains the graph of a function $u_\sigma$ for some $\sigma >  0$. In particular, we have that, with $t_0$ as before, $r(t_0) < \sqrt{2(n - 1)}$, and we argue as before that $\dot{r}(t) = 0$ for some $t \in (a, t_0)$.

Finally, if $(a, b) = \mathbb{R}$, then Theorem \ref{ThmConicalEnds} gives that $\gamma$ coincides with the line $r = \sqrt{2(n - 1)}$ for which (3) clearly holds.
 \end{proof}

\begin{proposition} \label{embedded_characterization}
Let $\gamma$ be a  complete solution to (\ref{ode_system}), such that one of the following statements hold:
\begin{itemize}
\item[(1)] $\gamma$  contains 7 vertical points.  
\item[(2)] $\gamma$ is closed and contains two vertical points in the first quadrant.
\item [(3)] $\gamma$ is not closed and contains one interior vertical point.
\end{itemize}
Then $\gamma$ is not embedded.
\end{proposition}

\begin{proof}[Proof of Proposition \ref{embedded_characterization}(1)]
 Consider  a segment of $\gamma$ containing seven consecutive vertical points, which we identify with the interval $[1, 7] \subset \mathbb{R}$ such that the vertical points correspond to integer values of the parameter. The vertical points will thus be denoted by $(x(k),r(k))$ for $k=1,\ldots,7$. Then by Lemma \ref{higgins},  after possibly reflecting $\gamma$ through the $r$-axis, we can assume that $x(k)$ is positive for $k$ odd and negative otherwise. Lemma \ref{maximal_graphs}(3) then gives  the existence of a horizontal point in each segment  $[k, k + 1]$, $k = 1, \ldots, 6$,  which we identify with the points $k+\fracsm{1}{2}$, $k =1, \ldots,6$. Lemma \ref{higgins}  implies that both the segments $[2+\fracsm{1}{2}, 3+\fracsm{1}{2}]$ and $[4+\fracsm{1}{2}, 5+\fracsm{1}{2}]$  intersect the line $r = \sqrt{2(n - 1)}$, so assume, after possibly reversing orientation  that $[2+\fracsm{1}{2}, 3+\fracsm{1}{2}]$ intersects to the left of $[4+\fracsm{1}{2}, 5+\fracsm{1}{2}]$. Take $\gamma_1$ to be the segment $[2+\fracsm{1}{2}, 3+\fracsm{1}{2}]$ and take $\gamma_2$ to be the segment $[3+\fracsm{1}{2}, 6+\fracsm{1}{2}]$. Note that on $\gamma_1$ the outward pointing unit tangent is $-e_x$ at each endpoint, while on $\gamma_2$ the outward pointing unit tangent is $e_x$. 
 
 We now translate the curve $\gamma_1$ in the positive $e_x$ direction until a point of first contact with $\gamma_2$. Note that such a point occurs,  since both segments intersect the line $r = \sqrt{2(n - 1)}$, with $\gamma_1$ intersecting to the left of $\gamma_2$, and that this point of first contact occurs away from the endpoints of both segments, and more generally does not occur at any horizontal point (since in particular the convexity near such a point is preserved under translation, it could not be a first intersection). Let $\hat{\gamma}_1 = \gamma_1 + c e_x $ denote the segment for which first point of contact occurs. Appealing to system (\ref{ode_system}) we get that
\begin{equation}
\dot{\theta}_{\gamma_2}(\hat{p}) - \dot{\theta}_{\hat{\gamma}_1}(\hat{p})= \frac{c}{2} \sin \theta 
\end{equation}
holds at the point of first contact $\hat{p}$, and where $\theta = \theta_{\gamma'}(\hat{p}) = \theta_{\hat{\gamma}_1}(\hat{p})$, which is a contradiction.
\end{proof}

\begin{proof}[Proof of Proposition \ref{embedded_characterization}(2)]
 For simplicity of description, we identify $\gamma$ with the unit circle $ \mathbb{S}^1 = \mathbb{R} / 2 \pi \mathbb{Z}$. Suppose now that there are two vertical points in $\gamma$ in the first quadrant, which after possibly reparametrizing we identify with the points $[0]$ and $[\pi]$ in $\mathbb{S}^1$. By assumption, we have that $x([0])$ and $x([\pi])$ are positive. Lemma \ref{higgins}  then gives an additional vertical point  on each arc $([0], [\pi])$ and $([\pi], [0])$, which  we identify with the points $[\frac{\pi}{2}]$ and $[ \frac{3 \pi}{2}]$ respectively. Now Lemma \ref{maximal_graphs}(3)  gives that there are horizontal  points along the arcs  $ ([0], [\frac{\pi}{2}])$, $([\frac{\pi}{2}], [\pi])$, $([\pi], [\frac{3 \pi}{2}])$ and $( [\frac{3 \pi}{2}], [0])$, which we identify with the four points $[\frac{(2k - 1)\pi}{4}]$, $k = 1, \ldots 4$. Lemma \ref{higgins} gives that both arcs $[[\frac{7 \pi}{4}], [\frac{\pi}{4}]]$ and $[[\frac{3 \pi}{4}], [\frac{5 \pi}{4}]]$ intersect the line $r = \sqrt{2(n - 1)}$, and after possibly relabeling, we can assume that $[[\frac{7 \pi}{4}], [\frac{\pi}{4}]]$ intersects to the left of  $[[\frac{3 \pi}{4}], [\frac{5 \pi}{4}]]$. Assume now that the arc $[[\frac{7 \pi}{4}], [\frac{\pi}{4}]]$ contains no vertical points other than $[0]$, and take  $\gamma_1 = [[\frac{7 \pi}{4}], [\frac{\pi}{4}]]$ and $\gamma_2 = [[\frac{\pi}{4}],[\frac{7 \pi}{4}]]$. We then translate $\gamma_1$ until a point of first contact with $\gamma_2$ and, arguing as in the proof of Proposition \ref{embedded_characterization}(1) obtain a contradiction.
\end{proof}  

 \begin{proof}[Proof of Proposition \ref{embedded_characterization}(3)]
Identify $\gamma$  with  an interval  $(a, b)$ under a Euclidean arc length parametrization,  and assume first that $a$ and $b$ are finite. Proposition \ref{embedded_characterization}(2) gives that $\gamma$ contains a finite number of vertical points. Lemma \ref{maximal_graphs}  then gives that $\gamma$ extends to the closed interval $[a, b]$ with vertical endpoints, and the assumption that $\gamma$ is complete in $H^+$ gives that these endpoints are  contained in the $x$-axis.

Now, suppose  $\gamma$ contains an interior vertical point  $c \in (a, b)$ , and assume that it is in the second quadrant. By Lemma \ref{maximal_graphs}(3) the arcs $[a, c]$ and $[c, b]$  each contain horizontal points $p_1$ and $p_2$, respectively, and consequently both intersect the line $r = \sqrt{2(n - 1)}$. Assume $[a, c]$ intersects to the  left of $[c, b]$, and assume $[p_1, p_2]$ contains no vertical points other than $c$. Then set $\gamma_1 = [a, p_1]$ and set $\gamma_2 = [p_1, b]$.

Note that $\gamma_1$ and $\gamma_2$ both intersect the line $r = \sqrt{2(n - 1)}$, and are compact. Moreover, the outward pointing tangent to $\gamma_1$ at $p_1$ is $-e_x$, and the outward pointing tangent to $\gamma_2$ at $p_1$ is $e_x$. As before, we translate $\gamma_1$ in the positive $e_x$ direction until a point of first contact with $\gamma_2$. By construction, this point of first contact cannot occur at $p_1$ (or its translated version). Moreover, by Lemma \ref{maximal_graphs}(2) it cannot occur at the endpoints of $\gamma_1$ and $\gamma_2$ contained in the $x$-axis. Hence, it is interior and non-transversal, and we obtain a contradiction as in the proof of Proposition \ref{embedded_characterization}(1) and (2).

Assume now that both $a$ and $b$ are infinite and identify $\gamma$ with the real line under a Euclidean arc length parametrization. Assume that $0$ is a vertical point in the second quadrant. Then by the completeness of $\gamma$, the arcs $(- \infty, 0]$, $[0, \infty)$ contain geodesic segments, maximally extended as graphs over the $x$-axis, and by Lemma \ref{maximal_graphs} both contain horizontal points $p_1$ and $p_2$, respectively. Assume as before that $[p_1, p_2]$ contains no vertical points other than $0$. By Proposition \ref{embedded_characterization}(2), $\gamma$ has a finite number of vertical points, and thus decomposes into a finite number of geodesic segments, maximally extended as graphs over the $x$-axis.  Then since $(-\infty, 0]$ and $[0, \infty)$  have  infinite Euclidean length, Lemma \ref{maximal_graphs} and Theorem \ref{ThmConicalEnds} imply that they contain the segments 
\begin{equation} \notag
\{(x, u_{\sigma_i}(x)) | x \geq 0 \},\quad i = 1, 2,
\end{equation}
 for distinct positive $\sigma_1$ and $\sigma_2$, respectively, after possibly reflecting through the $r$-axis. Thus, both $(- \infty, 0]$ and $[0, \infty)$ intersect the line $r = \sqrt{2(n - 1)}$ by Theorem \ref{ThmConicalEnds}, so assume that $(- \infty, 0]$ does so to the left of $[0, \infty)$. We then set $\gamma_1 := (-\infty, p_1]$ and $\gamma_2 = [p_1, \infty)$. As before $\gamma_1$ and $\gamma_2$ intersect the line $r = \sqrt{2(n - 1)}$, the outward pointing tangent to $\gamma_1$ at $p_1$ is $- e_x$, the outward pointing tangent to $\gamma_2$ at $p_1$ is $e_x$, and both curves $\gamma_1$ and $\gamma_2$ are properly embedded and separated by a positive distance (since $\sigma_1$ and $\sigma_2$ are distinct). We then translate $\gamma_1$ until a point of first contact with $\gamma_2$ and obtain a contradiction as in the previous case.
 
Finally, the case where the $(a, b) = [0, \infty)$, is handled exactly as in the previous cases, and  consequently we omit the details. 
    \end{proof}

We can now prove Theorem \ref{classification_theorem}.
 \begin{proof}[Proof of Theorem \ref{classification_theorem}]
 Note that by Proposition \ref{embedded_characterization} any non-closed embedded geodesic  $\gamma$ different from the $r$-axis cannot contain any interior vertical points and thus is globally given by the graph of a function $u(x)$ satisfying (\ref{SSODE}) on an open interval $I$ away from its endpoints. Let $\Sigma$ denote the surface of revolution determined by $\gamma$. Then $\Sigma$ is smooth and embedded and satisfies the self-shrinker equation (\ref{SSEq}). Lemma \ref{BlowUpLemma}  gives that $u'(x)x - u(x) < 0$, for all $x \in I$. This is in turn equivalent to the positivity of the mean curvature of $\Sigma$ with respect to the downward pointing unit normal (with respect to the axis of rotation). Huisken's classification of mean convex self-shrinkers \cite{Hu1} then implies that $\Sigma$ is either the round sphere of radius $\sqrt{2n}$, or the cylinder of radius $\sqrt{2(n - 1)}$. 
 
 If $\gamma$ is closed, then Proposition \ref{embedded_characterization} gives that it has at most two vertical points, and Lemma \ref{higgins} says that each is in a different quadrant of $H^+$. This concludes the proof.
   \end{proof}

\section{Appendix}
We include for completeness a proof of the smoothness of fixed points that we used in Lemma \ref{FixSmooth}.
\begin{proof}[Proof of Lemma \ref{FixSmooth}]
Let $\sigma$ be fixed, and let $x_{\sigma + h}, x_\sigma$ be fixed points for $\Phi_\sigma$ and $\Phi_{\sigma + h}$. Then
\begin{eqnarray} \notag
|x_{\sigma + h} - x_{\sigma}|  &= & |\Phi_{\sigma + h}(x_{\sigma + h}) - \Phi_\sigma(x_\sigma)| \\ \notag
& \leq & |\Phi_{\sigma + h}(x_{\sigma + h}) - \Phi_\sigma (x_{\sigma + h})| +  |\Phi_\sigma(x_{\sigma + h}) - \Phi_\sigma (x_\sigma)| \\ \notag
& \leq & \left|\frac{\partial \Phi_\sigma}{\partial\sigma} (\sigma, x_{\sigma + h})\right|h + \tau|x_{\sigma+ h} - x_\sigma| + o(h).
\end{eqnarray} 
This gives that the $x_\sigma$ are at least Lipshitz continuous functions of $\sigma$. To show differentiability, we again write
\begin{eqnarray*}
x_{\sigma + h} - x_{\sigma}  &= & \Phi_{\sigma + h}(x_{\sigma + h}) - \Phi_\sigma(x_\sigma) \\ \notag
& = & \Phi_{\sigma + h}(x_{\sigma + h}) - \Phi_\sigma (x_{\sigma + h}) +  \Phi_\sigma(x_{\sigma + h}) - \Phi_\sigma (x_\sigma) \\ \notag
& = & D_{x_\sigma}\Phi_\sigma (x_{\sigma + h} - x_\sigma) + O(|x_{\sigma + h} - x_\sigma|^2) +  \Phi_{\sigma + h}(x_{\sigma + h}) - \Phi_\sigma (x_{\sigma + h}).
\end{eqnarray*}
Rearranging terms, we see that
\begin{equation*}
\left(I - D_{x_\sigma} \Phi_\sigma - O(|x_{\sigma + h} - x_\sigma|)\right)(x_{\sigma + h} - x_{\sigma}) = \Phi_{\sigma + h}(x_{\sigma + h}) - \Phi_\sigma (x_{\sigma + h}).
\end{equation*}
Dividing by $h$ above and sending $h \rightarrow 0$, we get 
\begin{equation} \label{derivative_formula}
\frac{dx_\sigma}{d \sigma} = \big(I - D_{x_\sigma} \Phi_\sigma \big)^{-1}\frac{\partial \Phi}{\partial \sigma} (\sigma, x_\sigma).
\end{equation}
Note that the operator  $A = I - D_{x_\sigma} \Phi_\sigma$ is invertible, since the fact that $\Phi_\sigma$ is a contraction gives $||D_x \Phi||<1$.

Note that the formula for the derivative (\ref{derivative_formula}) gives that the fixed points $x_\sigma$ depend smoothly on the parameter $\sigma$, since the right hand side of $\sigma$ may be differentiated in $\sigma$ if the mappings $\Phi_\sigma$ are smooth.
\end{proof}

\bibliographystyle{amsalpha}

\end{document}